\documentclass[12pt]{amsproc}
\usepackage{amsmath,amsthm,mathtools,mathrsfs}
\usepackage{amsfonts,amssymb}
\newtheorem{thm}{Theorem}[section]
\newtheorem{cor}[thm]{Corollary}

\newtheorem{prop}[thm]{Proposition}

\newtheorem{defn}[thm]{Definition}
\theoremstyle{remark}
\newtheorem{rem}{Remark}[section]

\begin{document}

\title[Algebras of compactly supported continuous functions]
{On isomorphisms of algebras of compactly supported continuous
functions
}
\author{R. Lakshmi Lavanya}

\address{Indian Institute
of Science Education and Research Tirupati, Tirupati-517507}
\email{rlakshmilavanya@iisertirupati.ac.in}


\begin{abstract}
We study the general form of isomorphisms on the algebra of
compactly supported complex-valued continuous functions defined on
a locally compact Hausdorff space (the proof of which works for
the algebra of $\mathcal{C}^{(k)} -$differentiable functions on a
$\mathcal{C}^{(k)}-$ manifold as well). We obtain using only
topological techniques, that any such map is a composition of a
homeomorphism of the locally compact spaces (resp. $\mathcal{C}^{(k)}-$ diffeomorphism), and an automorphism
of the field of complex numbers. In the particular case when $X$
is a locally compact group, and the map preserves convolution
products, the resulting homeomorphism is also a group isomorphism.
An application of this gives a characterisation of the Fourier
transform on the algebra of Schwartz-Bruhat functions on locally
compact Abelian
groups.\\

\smallskip
\noindent \noindent \textit{2010 Mathematics Subject
Classification.} 54C05,54C40,22A10,43A25.\\
\textit{Keywords:} Continuous functions with compact support, Locally compact groups.
\end{abstract}
\maketitle
\section{Introduction}
\setcounter{equation}{0}
The fact that isomorphisms of algebras of functions defined over
certain topological spaces $X,$ taking values in a field
$\mathbb{F}$ determine the space X have been known for long since
the work of Gelfand and Kolmogoroff\cite{GK} and Milgram\cite{Mi}.
In 2005, Mr\v{c}un\cite{Mr} proved the following result for isomorphisms
of algebras of smooth functions. We denote $\mathbb{F}=\mathbb{R}$
or $\mathbb{C}.$
\begin{thm}(\cite{Mr})
For any Hausdorff smooth manifolds $M$ and $N$ (which are not
necessarily second-countable, paracompact or connected), any
isomorphism of algebras of smooth functions \\$T:
\mathcal{C}^\infty(N,\mathbb{F}) \rightarrow
\mathcal{C}^\infty(M,\mathbb{F})$ is given by composition with a
unique diffeomorphism $\tau: M \rightarrow N.$
\end{thm}
 Mr\v{c}un also proved an analogous result when the spaces
$ \mathcal{C}^\infty(N,\mathbb{F})$ and
$\mathcal{C}^\infty(M,\mathbb{F})$ are replaced by
$\mathcal{C}_c^\infty(N,\mathbb{F})$
 and $\mathcal{C}_c^\infty(M,\mathbb{F})$ respectively.\\

J. Grabowski(\cite{Gr2}) studied the above question for unital
algebras of functions on more general topological spaces.\\

 For a function $f:X\rightarrow \mathbb{F},$ we denote by $supp \ f,$ the support of $f,$ defined by $supp \ f := \ closure \ of \ \{x \in X: f(x)\neq 0\}.$

\begin{defn}(\cite{Gr2})
For a topological space $X,$ a subalgebra $\mathfrak{S}$ of
$\mathcal{C}(X,\mathbb{F})$ is called \textit{distinguishing} if
the following conditions hold good:
\begin{enumerate}
\item If $f\in \mathfrak{S}$ is nowhere vanishing, then $f^{-1}\in
\mathfrak{S},$ \item for each $p\in X,$ and every open
neighbourhood $U$ of $p,$ there ~exists $g\in \mathfrak{S}$ such
that $g(X) \subseteq [0,1], \ supp(g)\subset U,$ and $g(p')=~1$ if
and only if $p'=p.$
\end{enumerate}
\end{defn}

\begin{thm}(\cite{Gr2})
Let $\mathfrak{S}_i$ be a distinguishing algebra of
$\mathbb{F}-$valued continuous functions on a topological space
$X_i,\ i=1,2.$ Then, every algebra isomorphism $\Phi:
\mathfrak{S}_1 \rightarrow \mathfrak{S}_2$ is the pullback by a
homeomorphism $\varphi:X_2 \rightarrow X_1.$
\end{thm}

In \cite{Gr2}, the above result is initially proved for the algebra
$\mathcal{C}(X,\mathbb{F}),$ and then deduced for a distinguishing
subalgebra. The former result is proved using the maximal ideal
space theory for unital Banach algebras. In the present paper, we
obtain using topological techniques, the explicit form of the
algebra isomorphism, first for the algebra
$\mathcal{C}_c(X,\mathbb{F})$(which is not unital unless X is
compact). The corresponding result for the algebra
$\mathcal{C}^{(k)}_c -$ functions on a $\mathcal{C}^{(k)}-$
manifold follows as a consequence. Surprisingly, the proof of this
result does not make any use of the differentiable structure
on the underlying manifold.\\

Before stating our result, we discuss
a
few
more related results.\\

 In \cite{AAM4}, the authors obtained the
general form of a multiplicative bijection between certain classes
$\mathcal{B}$ of $\mathbb{F}-$valued functions.

\begin{thm}(\cite{AAM4})
Let M be a topological real manifold and $\mathcal{B}$ be either
$\mathcal{C}(M,\mathbb{C})$ or $\mathcal{C}_c(M,\mathbb{C}).$ Let
$T: \mathcal{B}\rightarrow \mathcal{B}$ be a multiplicative
bijection. Then there exists some homeomorphism $u: M \rightarrow
M$ and a function $p \in \mathcal{C}(M,\mathbb{C}), \ \Re(p)>0,$
such that either $T(re^{i\theta})(u(x)) = |r(x)|^{p(x)} \
e^{i\theta (x)}$ or $T(re^{i\theta})(u(x)) = |r(x)|^{p(x)} \
e^{-i\theta (x)}.$
\end{thm}

A similar result holds good for the algebras $\mathcal{C}^k$ of
$k-$times differentiable functions as well.

\begin{thm}(\cite{AAM4})
Let M be a $\mathcal{C}^k$ real manifold, $1\leq k\leq \infty,$
and $\mathcal{B}$ be one of the following function spaces:
$\mathcal{C}^k(M,\mathbb{C}),\mathcal{C}_c^k(M,\mathbb{C})$ or
$\mathcal{S}_\mathbb{C}(n).$ Let $T: \mathcal{B} \rightarrow
\mathcal{B}$ be a multiplicative bijection. Then there exists some
$C^k -$ diffeomorphism $u:M \rightarrow M$ such that either
$Tf(u(x)) = f(x)$ or $Tf(u(x)) = \overline{f(x)}.$
\end{thm}

The fundamental difference between the above results and the
present one is that of the spaces on which the function algebras
are defined. In the former case, these are differentiable
manifolds. Hence the proofs also involved ideas on jet bundles of
functions which which are not available in the current setting of
locally compact Hausdorff spaces. Also, the algebra of smooth
functions is replaced with that of compactly supported continuous
functions defined on a locally compact Hausdorff space. However, the proof of Theorem \ref{CCXY} gives a \textit{unified} proof for the description of isomorphisms on the algebras of $\mathcal{C}_c^{(k)}-$differentiable functions on $k-$differentiable manifolds, for all $1 \leq k\leq \infty.$ \\

We denote by $\mathcal{C}_c(+,\cdot,\ast),$ the algebra of all
compactly supported complex-valued functions on $X,$ under
pointwise addition and multiplication.\\

Our main result is
\begin{thm}\label{CCXY} Let $X$ and $Y$ be locally compact Hausdorff spaces. Let $T:(\mathcal{C}_c(X),+,\cdot) \rightarrow
(\mathcal{C}_c(Y),+,\cdot)$ be an isomorphism of algebras. Then there
exists a homeomorphism $\psi: Y \rightarrow X$ such that either
$Tf = f\circ\psi$ for all $f\in \mathcal{C}_c(X),$ or $Tf =
\overline{f\circ \psi}$ for all $f\in \mathcal{C}_c(X).$
\end{thm}

\noindent \begin{rem}Apriori, the isomorphism is defined only on
functions in $\mathcal{C}_c(X,\mathbb{F}).$ But as illustrated in
the proof of the above theorem, the isomorphism has a canonical
extension to the algebra of constant functions. This extension
yields the explicit form of the isomorphism using ideas only from
point-set topology.
\end{rem}

In order to prove Theorem \ref{CCXY}, first we study the general structure of the map $T.$\\

For $x_0\in X,$ define
$$S(x_0): = \{f \in \mathcal{C}_c(G): x_0\in Supp\ f\}.$$

\noindent\begin{prop} \label{phi} Let $X$ and $Y$ be locally
compact Hausdorff spaces. Let $T:\mathcal{C}_c(X) \rightarrow
\mathcal{C}_c(Y)$ be a multiplicative bijection. For any $x_0 \in
X,$ there exists $y_0 \in Y$ such that $Tf\in S(y_0)$ whenever $f
\in S(x_0).$
\end{prop}

\begin{proof}
First, we observe that for functions $f,g \in \mathcal{C}_c(X),$ we have $Tg=1$ on $supp \ Tf$ whenever $g=1$ on $supp\ f.$\\

The condition $g=1$ on $supp \ f$ gives $f\cdot g=f.$ Since $T$ preserves products, this gives $Tf\cdot Tg = T(f\cdot g) =Tf,$ which guarantees $Tg=1$ wherever $Tf$ is nonzero.

Let $y_0 \in supp\ Tf$ be such that $Tf(y_0)=0.$ This gives a net $(y_\tau)$ converging to $y_0$ such that $Tf(y_\tau) \neq 0$ for any $\tau.$ By the above argument, $Tg(y_\tau)=1$ for all $\tau,$ and hence $Tg(y_0)=1.$\\

Let $E:=\{f\in \mathcal{C}_c(X):
f(x_0) \neq 0\}.$ \\

Fix $g\in E.$ Then $K:= supp(Tg)$ is compact. For $f\in E,$ let $K(f): = K \cap \ supp (Tf).$ For any finite collection $\{f_0:=g,f_1,\cdots , f_k\}$ in $E,$ the product
$\prod_{j=0}^k \ f_j\not \equiv 0.$ This gives by the multiplicativity of the map $T,$ that $T(\prod_{j=0}^k \ f_j) =\prod_{j=0}^k \ Tf_j
 \not \equiv 0.$ This ensures $\cap _{j=0} ^k
\ K_{f_j} \neq \emptyset.$ In otherwords, the collection $\{K_f:f\in E\}$ of closed subsets of the compact set $K$ satisfies finite intersection property. This guarantees  the existence of some element $y_0 \in \bigcap \limits_{f\in E} \ K_f
\neq \emptyset.$ \\

\noindent \textbf{Claim.} For any function $f\in \mathcal{C}_c(X),$ we have $Tf\in S(y_0)$ whenever $f\in S(x_0).$\\

\noindent \textit{Proof of Claim.} We prove the claim in two cases.\\

\noindent \textbf{Case 1.} $f(x_0) \neq 0.$\\

Choose $g \in \mathcal{C}_c(X)$ with $f\cdot g=1$ on a neighbourhood $ V$ of $x_0.$ Let $h\in \mathcal{C}_c(X)$ be such that $h=1$ on a neighbourhood $W$ of $x_0$ and $supp \ h \subseteq V.$ Now, $f\cdot g=1$ on $supp \ h$ gives $Tf\cdot Tg =1$ on $supp\ Th,$ which contains $y_0.$ Thus $Tf\in S(y_0).$\\

\noindent \textbf{Case 2.} $f(x_0) =0.$\\

Suppose $Tf \not \in S(y_0).$ Let $W$ be a neighbourhood of $y_0$ on which $Tf$ vanishes identically. Choose $h\in \mathcal{C}_c(Y)$ with $supp\ h\subseteq W$ and $h(y_0)\neq 0.$ For $g \in \mathcal{C}_c(X)$ with $Tg=h,$ we have
$$0 \equiv Tf\cdot h = Tf\cdot Tg = T(f\cdot g).$$
This means $f\cdot g \equiv 0, $ which is not possible as $g(x_0) \neq 0$ and also $x_0\in supp \ f.$\\
\end{proof}

Let $\varphi:X\rightarrow Y$ be defined as follows:
$$\varphi(x_0) =y_0, \ \textrm{ \  if \ } Tf\in S(y_0) \textrm{ \ for \ all \ } f\in S(x_0).$$

\begin{prop}\label{homeo}
The map $\varphi:X\rightarrow Y$ is a homeomorphism.
\end{prop}

\begin{proof}
  First we ensure that the map is well-defined. Suppose there exists $x_0\in X$ with $\varphi(x_0)=y_1 , \ \varphi(x_0)=y_2$ and $y_1 \neq y_2.$ Let $g_1,g_2 \in \mathcal{C}_c(Y)$ be supported in disjoint neighbourhoods $V_1$ and $V_2$ of $y_1$ and $y_2,$ respectively, with $g_1(y_1)
\neq 0$ and $g_2(y_2) \neq 0.$ Then for $f_1$ and $f_2$ such that $Tf_1=g_1$ and $Tf_2 =g_2,$ we have
  $$0 \equiv g_1\cdot g_2 =Tf_1 \cdot Tf_2 =T(f_1\cdot f_2),$$
  and hence $f_1\cdot f_2 \equiv 0.$ This is in contradiction with $(f_1 \cdot f_2)(x_0) \neq 0.$\\

  Repeating the above arguments for $T^{-1}$ instead of $ T$ gives that $\varphi$ is a bijection.\\

  Suppose $\varphi$ is not continuous at some point $x_0.$ Let $(x_\tau)$ be a net converging to $x_0 \in X$ with $\varphi(x_\tau)$ not converging to $\varphi(x_0).$ Let $h\in \mathcal{C}_c(Y)$ be such that $h(\varphi(x_0))=1$ and $supp \ h \subseteq V,$ for a neighbourhood $V$ of $\varphi(x_0)$ which does not contain $\varphi(x_\tau)$ for any $\tau.$ For the function $g\in \mathcal{C}_c(X)$ such that $Tg=h,$ we have $g(x_\tau)=0$ for all $\tau,$ and hence $g(x_0)=0,$ which contradicts $Tg(\varphi(x_0))=1.$ A similar argument for $T^{-1}$ gives that $T$ is a homoemorphism.\\
\end{proof}

We are now ready to prove Theorem \ref{CCXY}.\\

 \noindent \textit{PROOF OF THEOREM \ref{CCXY}.} By Propositions \ref{phi} and \ref{homeo}, we infer that
the spaces X and Y are homeomorphic.\\

Though the map $T$ is apriori defined only on $\mathcal{C}_c(X),$ one
can extract information as to how $T$ acts on the constant
functions on $X,$ which
we denote just by the constant itself.\\

For $f,g \in \mathcal{C}_c(X),$ and $\alpha (\neq 0)\in
\mathbb{C},$ we have
$$ T(\alpha f)(y)\ Tg(y) = T(\alpha fg)(y) = T(f)(y) \ T(\alpha g)(y) , \ y\in Y.$$
For $y\in Y,$ let $h \in \mathcal{C}_c(X)$ be such that $Th(y)
\neq 0.$ Then we have \noindent \begin{eqnarray*}
 T(\alpha f)(y)  &=& \frac{T(\alpha h)(y)}{Th(y)} \ {Tf}(y) \textrm{ \ for \ all \ } f\in \mathcal{C}_c(X)\\
   &=& m(\alpha,y) \ Tf(y) \textrm{ \ (say)}.
\end{eqnarray*}
Thus $T(\alpha f)(y) = m(\alpha ,y)\ Tf(y),$ for all $y\in Y.$ By
definition, the function $m(\cdot,\cdot)$ is continuous in the
second variable, as a function
of $y\in Y.$\\

For $y \in Y,$ choose $f\in \mathcal{C}_c(Y)$ such that $Tf(y)
\neq 0.$ Then
 $$Tf(y) = T(1 \cdot f)(y) = m(1,y) Tf(y).$$
Varying $y$ over $Y$ and appropriately varying $f,$ we get
$m(1,\cdot)\equiv 1.$\\ A similar argument gives $m(0,\cdot) \equiv
0.$\\

\noindent For $\alpha,\beta \in \mathbb{C},$ and $f\in \mathcal{C}_c(X)$ such that $Tf(y)\neq 0,$ we have
$$m(\alpha+\beta,y) Tf(y) = T((\alpha+\beta)f)(y) = [m(\alpha,y)+m(\beta,y)] Tf(y).$$
Also,
$$m(\alpha\beta,y)Tf(y) = T(\alpha \beta f)(y) = m(\alpha,y) T(\beta f)(y) = m(\alpha,y) m(\beta,y) Tf(y).$$
Thus $m(\cdot,\cdot)$ is additive and multiplicative in the first variable.\\

Suppose $f\in \mathcal{C}_c(X)$ with $f=c(\neq 0)$ near $x_0.$ Then $\frac{f}{c} = 1$ near $~x_0$ and so $Tf(\cdot) = m(c,\cdot)$ near $y_0(=\varphi(x_0)).$
In particular, we have
$$Tf(y_0) = m(f(x_0),y_0).$$

Next, to find the action of $T$ on a general function in $\mathcal{C}_c(X).$ Fix $g\in \mathcal{C}_c(X)$ and $x_0\in X.$\\

\noindent  If $g(x_0)=0,$ then we have
 $$Tg(y_0)=Tg(\varphi(x_0)) = 0 = m(0,y_0).$$\\
\noindent Suppose $g(x_0)\neq 0.$ For a function $f\in \mathcal{C}_c(X)$ with $f=g(x_0)$ near $x_0,$ we have
$$0=T(f-g)(y_0) = Tf(y_0) - Tg(y_0) = m(f(x_0),y_0) - Tg(y_0) = m(g(x_0),y_0) - Tg(y_0).$$
Working with appropriate locally constant functions $f,$ we get

\begin{eqnarray}
 Tg[\varphi(x)]=&  m(g(x),\varphi(x)) \textrm{ \ for \  all \ } g\in C_c(X). \label{tform}
\end{eqnarray}
 Since all the other functions involved in the above equation are continuous, we get $m(\cdot,\cdot)$ is continuous in both the variables. \\

We have so far observed that for any $y\in Y,$ the map $m(\cdot,y)$ is a continuous additive and multiplicative automorphism of $\mathbb{C}, $ and hence we get $m(\alpha,y) =\alpha$ or $m(\alpha,y)
= \overline{\alpha}$ for all $\alpha\in \mathbb{C}.$ Since $m(\cdot,\cdot)$ is continuous in the second variable, we are left with $m(\alpha,\cdot)\equiv \alpha$ or $m(\alpha,\cdot) \equiv \overline{\alpha}.$\\

 By definition, the map $m(\cdot,\cdot)$ is
canonical to the map $T.$ Hence by abuse of notation, we may define
$T(\alpha):=m(\alpha,y_0)$ for $\alpha \in \mathbb{C},$ and a
fixed $y_0 \in Y.$ Here we emphasise that apriori, the
map $T$ was defined on $\mathcal{C}_c(X),$ which does not contain
the constant functions unless $X$ is compact.\\

Combining the above observations with Equation (\ref{tform}),we get
$$
Tf[\varphi (x)] = T(f(x)), \textrm{ \ for \ all \ }f\in \mathcal{C}_c(X).$$

\noindent Since $\psi =\varphi ^{-1},$ we get that either
$Tf(x) = f[\psi (x)]$ or $Tf(x) = \overline{f[\psi (x)]}.$\\
\qed\\

The proof of the above result provides a unified proof of the description of algebra isomorphisms of
$\mathcal{C}^{(k)}-$ functions on a real $\mathcal{C}^{(k)}-$differentiable manifold, for $0\leq k\leq \infty$. \\

\begin{cor}
Let $M$ and $N$ be real $\mathcal{C}^{(k)}-$manifolds for some $k,
\ 0\leq k\leq \infty.$ Let $T:(\mathcal{C}^{(k)}_c(M),+,\cdot)
\rightarrow (\mathcal{C}^{(k)}_c(N),+,\cdot)$ be an isomorphism of
algebras. Then
there exists a $\mathcal{C}^{(k)}-$ diffeomorphism $\psi: N
\rightarrow M$ such that either $Tf = f\circ\psi$ for all $f\in
\mathcal{C}^k_c(M),$ or $Tf = \overline{f\circ \psi}$ for all
$f\in \mathcal{C}^k_c(M).$
\end{cor}
\begin{proof}
Follows from the proof of Theorem \ref{CCXY}. The conclusion that
the map $\psi$ is a $\mathcal{C}^{(k)}-$ diffeomorphism follows
from the explicit form of the algebra isomorphism.
\end{proof}

 Theorem \ref{CCXY} proves that the topology of X is completely
determined by the algebraic structure of
$(\mathcal{C}_c(X),+,\cdot).$ It is natural to ask if
$\mathcal{C}_c(X)$ carries any information on the
\textit{algebraic}
structure of $X.$\\

Suppose $X$ is a locally compact group. Then there is a unique
left Haar measure $\mu$ on $X.$ For functions $f,g \in
\mathcal{C}_c(X),$ the convolution product $f\ast g$ is defined as
$$f\ast g (x) = \int \limits_G f(x-y) \ g(y) \ d\mu(y).$$
Equipped with the convolution product, the
space $\mathcal{C}_c(X)$ forms an algebra, which we denote by
$(\mathcal{C}_c(X),+,\ast).$  Theorem \ref{CCXY} states that the
algebra $(\mathcal{C}_c(X),+,\cdot)$ determines the topology of
$X.$\\

It is natural to ask the following questions:

\begin{enumerate}
\item Does $\mathcal{C}_c(X)$ also determine the algebraic
structure of X? \item Does the explicit form of such an
isomorphism depend on the structure of the field $\mathbb{F}?$
\end{enumerate}

Our next result gives a positive answer to the above questions in
the
particular cases when $\mathbb{F}=\mathbb{R}$ or $\mathbb{C}.$\\

\begin{thm}\label{LAGH} Let $G$ and $H$ be locally compact groups. Let $T:\mathcal{C}_c(G) \rightarrow \mathcal{C}_c(H)$ be an
algebra isomorphism under pointwise and convolution
products. Then
there exists a multiplicative homeomorphism $\psi: H \rightarrow
G$ such that either $Tf(y) = f(\psi(y))$ for all $f\in
\mathcal{C}_c(G),$ or $Tf(y) = \overline{f(\psi(y))}$ for all
$f\in \mathcal{C}_c(G).$
\end{thm}

\begin{proof}
In view of Theorem \ref{CCXY}, it remains to prove that the map
$\psi^{-1} = \varphi:G \rightarrow H$ is a
homomorphism.\\

Suppose $\varphi (xy) \neq \varphi (x) \varphi (y)$ for some $x,y\in G.$\\

We denote $(\varphi \otimes \varphi)(x,y) = \varphi(x) \varphi(y), \ x,y\in G.$ Let $p_G$ and $p_H$ denote the product maps of the respective groups.\\

Suppose $(\varphi \circ p_G) (x,y)\neq [p_H \circ (\varphi \otimes
\varphi)](x,y)$ for some $x,y \in G.$ Let $V_{xy}$ and
$V_{x\diamond y}$ be disjoint neighbourhoods of $(\varphi\circ
p_G)(x,y)$ and $[p_H\circ (\varphi \otimes \varphi)](x,y),$
respectively. Since the maps $(\varphi\circ p_G)$ and $p_H\circ
(\varphi \otimes \varphi)$ are continuous on $G\times G,$ we get
neighbourhoods $W_{x,1},\ W_{y,1}$ and $W_{xy,1}$ of $x,\ y$ and
$xy,$ respectively, such that
$$\varphi (W_{xy,1}) \subseteq V_{xy} \textrm{ \ and  \ } W_{x,1}W_{y,1} \subseteq
W_{xy,1}.$$
\begin{eqnarray} \vspace{-1cm} \textrm{Hence} \hspace{2.3cm} \varphi
(W_{x,1} W_{y,1}) &\subseteq& \varphi (W_{xy,1}) \subseteq
V_{xy}.\hspace{3cm}\nonumber\label{eq1}
\end{eqnarray}

\noindent Similarly, continuity of the map $[p_H \circ (\varphi
\otimes \varphi)],$ gives rise to neighbourhoods $W_{x,2},
W_{y,2}, V_{x,2}$ and $V_{y,2}$ of $x, y,   \varphi(x)$ and
$\varphi(y),$ respectively, such that \vspace{-0.3cm}
$$V_{x,2}\ V_{y,2}\subseteq V_{x\diamond y
}, \ \varphi (W_{x,2})\subseteq V_{x,2} \textrm{ \  and \ }
\varphi (W_{y,2}) \subseteq V_{y,2}.$$

Let $W_x= W_{x,1}\cap W_{x,2}, \ W_y = W_{y,1}\cap W_{y,2}.$
We have \noindent \begin{eqnarray}
  \varphi (W_x)& \subseteq & \varphi (W_{x,2}) \subseteq V_{x,2} =V_x \ (say) \nonumber \\
   \varphi (W_y)& \subseteq & \varphi (W_{y,2}) \subseteq V_{y,2} =V_y \ (say)\nonumber \\
   \varphi (W_x W_y)& \subseteq & \varphi (W_{x,1} W_{y,1})  \subseteq \varphi (W_{xy,1}) \subseteq V_{xy} \label{eq3}
\end{eqnarray}
Let nonzero functions $f_x,f_y \in \mathcal{C}_c(G)$ be such that $supp \ f_x
\subseteq W_x$ and $supp\ f_y \subseteq W_y.$ Let $g_x = Tf_x$ and $g_y = Tf_y.$ Then $T(f_x\ast f_y)
= g_x \ast
g_y \not\equiv ~0.$\\

\noindent We have \vspace{-0.3cm}
\noindent \begin{eqnarray}
 \nonumber supp(g_x \ast g_y) &=&supp \ T(f_x \ast f_y)  \subseteq  (supp \ Tf_x) \ (supp \ Tf_y) \\
\nonumber &= & \varphi (supp \ f_x) \ \varphi (supp \ f_y)
\subseteq \varphi( W_x) \ \varphi (W_y)
\\ & \subseteq & V_{x,2}\ V_{y,2} \subseteq V_{x\diamond y } \label{eq4}
    \end{eqnarray}
But $supp \ (f_x  \ast f_y) \subseteq (supp \ f_x) \ (supp \ f_y)
\subseteq W_x\ W_y.$ By $(\ref{eq3}),$ this gives
 \noindent \begin{eqnarray}
 \nonumber supp(g_x \ast g_y) &= &supp (Tf_x \ast Tf_y) = supp \ T(f_x \ast f_y)\\
  & =& \varphi (supp (f_x  \ast f_y)) \subseteq  \varphi (W_x \ W_y)\subseteq
 V_{xy}. \label{eq5}
 \end{eqnarray}
From $(\ref{eq4})$ and $(\ref{eq5})$, we get $$supp(g_x \ast g_y)
 \subseteq V_{x\diamond y } \cap V_{xy} = \emptyset.$$ This gives $g_x\ast
g_y \equiv 0,$ a contradiction to $f_x\ast f_y\not\equiv 0$. Hence the map $\varphi$ is a homomorphism.\\
\end{proof}

\begin{cor} \label{G=H}
In Theorem \ref{LAGH}, in addition to the hypotheses, if $G=H,$
then the map $\varphi$ is measure-preserving.
\end{cor}

\begin{proof} We denote $\psi:=\varphi^{-1}.$
Suppose $Tf = f\circ \psi$ for all $f\in \mathcal{C}_c(G).$ For
the identity element $e$ of $G,$ we have $\psi(e) = e,$ and
hence\noindent
\begin{eqnarray*}
 (f\ast g)(e) = T(f\ast g)(e)  &=& (Tf \ast Tg)(e) \\
   &=& \int\limits_{G} Tf(y^{-1} ) \ Tg(y) \ d\mu(y) \\
   &=& \int\limits_{G} f[\psi(y^{-1} )] \ g[\psi (y)] \ d\mu(y)\\
      &=& \int\limits_{G} f[\psi(y)^{-1} ] \ g[\psi (y)] \
      d\mu(y)\\
     \textrm{i.e., \   } \int\limits_{G} f(y^{-1}) \ g(y) \ d\mu(y)  &=& \int\limits_{G} f(y^{-1}) \ g(y) \ d\mu(\varphi(y))
\end{eqnarray*}
For $f\in \mathcal{C}_c(G),$ choosing $g\in \mathcal{C}_c$ with
$g=1$ on $(supp \ f)^{-1},$ we get
$$\int \limits_G f(y^{-1}) \ d\mu(y) =\int\limits_G f(y^{-1}) \
d\mu(\varphi(y)).$$ Thus we have
$$\int\limits_G f(y) \ d\mu(y) = \int\limits_G (f\circ \psi)(y) \ d\mu(y)$$
for all functions $f\in \mathcal{C}_c(G).$ Hence the map
$\psi$ is measure-preserving.\\

A similar argument applies when $T$ is of the form $Tf =
\overline{f\circ \psi},$ proving
our result.
\end{proof}

We now recall the definition of the Schwartz-Bruhat space of
functions on a locally compact Abelian group.\\

Let $G$ be a locally compact Abelian group with unitary dual $\Gamma.$ We denote by $dx,$ the Haar measure on $G.$ For an integrable function $f$ on $~G,$ its Fourier transform is defined as
$$\widehat{f}(\xi) = \int \limits_G f(x) \ \overline{\langle x, \xi \rangle} \ dx.$$
In the above, $\langle x, \xi \rangle$ stands for the dual action of $\xi \in \Gamma$ on $x \in G.$

For the function $f^*(x) = \overline{f(x^{-1})},$ we have $\widehat{[f^*]}(\xi) = [\widehat{f}(\xi)]^*$ for all $\xi \in \Gamma.$\\

F. Bruhat\cite{Br} extended the notion of a smooth function to a
large class of groups, which encompasses the LCA groups. 1n 1975,
M.S. Osborne characterised the Schwartz-Bruhat space of functions
on a LCA group in terms of the asymptotic behavior of the function
and its Fourier transform.\\

 A function $f: G\rightarrow \mathbb{C}$ is said to
 belong the \textit{Schwartz-Bruhat space} if $f$ satisfies the following
 conditions:
 \noindent \begin{description}
 \item[(a)] $f\in \mathcal{C}^\infty(G).$
 \item[(b)] $P(\partial)f\in L^\infty(G)$ for all polynomial differential
 operators $P(\partial),$ where the polynomial is in $\mathbb{R}^n \times
 \mathbb{Z}^k$
 variables.
 \end{description}
 For any locally compact Abelian group $G,$ we have
 $$\mathcal{S} (G) = \lim \limits_{\rightarrow} \ \mathcal{S} (H/K),$$
 where the direct limit is taken over all pairs $(H,K)$ of subgroups
 of $G$ such that
 \noindent \begin{description}
    \item[(i)] The subgroup $H$ is open and compactly generated,
    \item[(ii)] The subgroup $K$ is compact
    \item[(iii)] The quotient $H/K$ is a Lie group.
 \end{description}

 For a more detailed definition of the space $\mathcal{S} (G),$ we refer the
 reader to \cite{LASB} and \cite{Wa}. The following result gives a complete description of the Fourier transform on the Schwartz-Bruhat space of functions.
 \noindent \begin{thm}(\cite{Wa})
 The Fourier transform maps $\mathcal{S} (G)$ isomorphically onto
 $\mathcal{S} (\Gamma)$ and $\mathcal{S} (G)$ is dense in $L^1(G).$
 \end{thm}

 Let $\mathcal{C}_c^\infty  = \mathcal{C}_c^\infty (G): = \{f\in \mathcal{S} (G) : supp \ f \textrm{ \ is  \ compact}\}.$\\

The following results were obtained in \cite{LASB}. However, we
include short versions of their proofs here to illustrate how they
are related to our main results.

 \noindent
\begin{thm}(\cite{LASB}) Let $G$ be a locally compact Abelian
group, and $\Gamma,$ its unitary dual. Let $T:\mathcal{S}(G)
\rightarrow \mathcal{S}(\Gamma)$ be a bijection such that for all
functions $f,g \in \mathcal{S}(G),$ we have \noindent
\begin{description} \item[(a)]  $T(f+g^*) = T(f)+[T(g)]^*,$
 \item[(b)] $T(f\cdot g)=
T(f) \ast T(g )$, \item[(c)] $T(f\ast g)= T(f)\cdot T(g).$
\end{description} Then there exists a measure-preserving multiplicative homeomorphism $\psi$ of $G$ onto itself such that
either $Tf = \widehat{(f\circ \psi)}$ for all $f\in
\mathcal{S}(G),$ or $Tf = \widehat{(\overline{f\circ \psi)}}$ for
all $f\in \mathcal{S}(G).$
\end{thm}

\noindent \begin{proof} Let $f\in \mathcal{S}(G).$ As the Fourier
transform maps $\mathcal{S}(G)$ bijectively onto
$\mathcal{S}(\Gamma),$ we get that $Tf=\widehat{g}$ for a unique
function $g\in \mathcal{S}(G).$ Define $Uf:=g.$ Then the map $U$
maps $\mathcal{S}(G)$ bijectively onto itself, such that for all
functions $f,g\in \mathcal{S}(G),$ we have \noindent
\begin{enumerate} \item $U(f+\overline{g}) =
U(f)+\overline{U(g)},$ \item $U(f\cdot g)= U(f) \cdot U(g )$,
\item  $U(f\ast g)= U(f)\ast U(g).$
\end{enumerate}
The conclusion is a consequence of the following result for the
map $U.$
\end{proof}

\begin{thm}\label{SGG}(\cite{LASB}) Let $G$ be a locally compact
Abelian group. Let $U:\mathcal{S}(G) \rightarrow \mathcal{S}(G)$
be a bijection satisfying the following conditions for all
functions $f,g \in \mathcal{S}(G):$ \noindent
\begin{enumerate} \item $U(f+\overline{g}) = U(f)+\overline{U(g)},$ \item
$U(f\cdot g)= U(f) \cdot U(g)$, \item $U(f\ast g)= U(f)\ast U(g).$
\end{enumerate} Then there exists a measure-preserving multiplicative homeomorphism $\psi: G \rightarrow G$ such that
either $Uf = f\circ \psi$ for all $f\in \mathcal{S}(G),$ or $Uf =
\overline{f\circ \psi}$ for all $f\in \mathcal{S}(G).$
\end{thm}

\noindent \begin{proof} For $x_0\in G,$ define
$$S(x_0):= \{f\in \mathcal{S}(G) : x_0 \in supp \ f\}.$$
Let $f,g\in \mathcal{S}(G).$ As in the beginning of the proof of Theorem \ref{CCXY}, we have that $Ug=1$ on $supp(Uf)$ whenever $g=1$ on $supp \ f.$\\

\noindent \textbf{Claim.} If $f\in \mathcal{C}_c(G),$ then $Uf \in \mathcal{C}_c(G).$\\

\noindent Proof of Claim. Let $f\in \mathcal{C}_c(G)$ with $f(x_0) \neq 0.$ Let $g\in \mathcal{S}(G)$ be such that $g=1$ on $supp\ f.$ Then $Ug=1$ on $supp(Uf).$ As $Ug\in \mathcal{S}(G),$ we get that $supp(Uf)$ is compact.\\

We observe that the proof of Corollary \ref{G=H} applies to
functions in $\mathcal{S}(G),$ thus proving the result.
\end{proof}


\noindent \begin{thebibliography}{99}
\bibitem{AAM3} S. Alesker, S. Artstein-Avidan, D. Faifman, V. Milman, \textit{A characterization product preserving
maps with applications to a characterization of the Fourier
tranform}, Ill. J. Math. (3) \textbf{54} (2010), 1115 - 1132.

\bibitem{AAM4}  S. Artstein-Avidan, D. Faifman, V. Milman, \textit{On multiplicative
 maps of continuous and smooth functions,} Geometric aspects of functional analysis, Lecture Notes in Math. \textbf{2050}, Springer, Heidelberg (2012), 35-59.

\bibitem{Br} F. Bruhat, \textit{Distributions sur un groupe localement
compact et applications \`{a} l'\'{e}tude des repr\'{e}sentations
des groupes p-adiques,} Bull. Soc. Math. France \textbf{89}
(1961), 43-75.

\bibitem{Fo} G.B. Folland, \textit{A Course in Abstract Harmonic Analysis,} Studies in Advanced
Mathematics, CRC Press, Boca Raton, FL, 1995.

\bibitem{GK} I. Gelfand, A. Kolmogoroff, \textit{On rings of continuous
functions on topological spaces,} C. R. (Dokl.) Acad. Sci. URSS
\textbf{22} (1939), 11-15.

\bibitem{Gr2} J. Grabowski, \textit{Isomorphisms of algebras of smooth functions revisited,} Arch. Math. (Basel) (2) \textbf{85} (2005), 190-196.

\bibitem{LASB} R. Lakshmi Lavanya, \textit{Revisiting the Schwartz-Bruhat
space on locally compact Abelian groups,} submitted.

\bibitem{Mi} A.N. Milgram, \textit{Multiplicative semigroups of continuous
functions,} Duke Math J. \textbf{16} (1949), 377-383.

\bibitem{Mr}  J. Mr\v{c}un, \textit{On isomorphisms of algebras of smooth functions,} Proc. Amer. Math. Soc. (10) \textbf{133} (2005),
3109-3113.

\bibitem{MS} J. Mr\v{c}un, P. \v{S}emrl, \textit{Multiplicative bijections between
algebras of differentiable functions,} Ann. Acad. Sci. Fenn. Math.
(2) \textbf{32} (2007), 471-480.

\bibitem{SS} F.C. S\'{a}nchez, J.C. S\'{a}nchez, \textit{Some preserver problems
on algebras of smooth functions,} Ark. Mat. \textbf{48} (2010),
289-300.

\bibitem{Os} M. Scott Osborne,\textit{ On the Schwartz-Bruhat Space and
the Paley-Wiener theorem for locally compact Abelian groups,} J.
Funct. Anal. \textbf{19} (1975), 40-49.

\bibitem{Wa} A. Wawrzy\'{n}czyk, \textit{On tempered distributions and
Bochner-Schwartz theorem on arbitrary locally compact Abelian
groups,} Colloq. Math. (2) \textbf{19} (1968), 305-318.

\bibitem{We} A. Weil, \textit{Sur certain groupes d'op\'{e}rateurs unitaires,}
Acta Math. \textbf{111} (1964), 143-211.
\end{thebibliography}
\end{document}